\documentclass[12pt]{amsart}
\usepackage{amssymb} % mathematical fonts (in amsfonts)
\usepackage[mathscr]{eucal} % load euler fonts (in amsfonts)
\usepackage{amsmath} %this is needed for \rxar (\xrightarrow)
\usepackage{epsfig}
\usepackage{mathabx}
\usepackage{graphicx}
\usepackage{amscd}
\usepackage{mathrsfs}
\usepackage{hyperref}
\hypersetup{
    colorlinks=true,
    linkcolor=blue,
    filecolor=magenta,      
    urlcolor=cyan,
}
\usepackage[margin=0.8in]{geometry}

\DeclareMathOperator{\pd}{proj.dim}
\DeclareMathOperator{\id}{inj.dim}
\DeclareMathOperator{\flatdim}{flat.dim}
\DeclareMathOperator{\LH}{LH}
\DeclareMathOperator{\silp}{silp}
\DeclareMathOperator{\spli}{spli}
\DeclareMathOperator{\splif}{spli-finiteness}
\DeclareMathOperator{\sflilf}{sfli-local-finiteness}
\DeclareMathOperator{\splilf}{spli-local-finiteness}
\DeclareMathOperator{\GInj}{GInj}
\DeclareMathOperator{\GProj}{GProj}
\DeclareMathOperator{\Hom}{Hom}
\DeclareMathOperator{\hh}{H}
\DeclareMathOperator{\first}{(a)}
\DeclareMathOperator{\ditio}{(b)}
\DeclareMathOperator{\tritio}{(c)}
\DeclareMathOperator{\Gcd}{Gcd}
\DeclareMathOperator{\tor}{Tor}
\DeclareMathOperator{\GFlat}{GFlat}
\DeclareMathOperator{\Ij}{inj}
\DeclareMathOperator{\sfli}{sfli}
\DeclareMathOperator{\silf}{silf}
\DeclareMathOperator{\WChar}{WChar}
\DeclareMathOperator{\Gpd}{Gpd}
\DeclareMathOperator{\Gid}{Gid}
\DeclareMathOperator{\Ghd}{Ghd}
\DeclareMathOperator{\fgidim}{FinGid}
\DeclareMathOperator{\supr}{sup}
\DeclareMathOperator{\Gfd}{Gfd}
\DeclareMathOperator{\Findim}{FinProjDim}
\DeclareMathOperator{\ffdim}{FinFlatDim}
\DeclareMathOperator{\fidim}{FinInjDim}

\DeclareMathOperator{\Ext}{Ext}

\DeclareMathOperator{\Coind}{Coind}
\DeclareMathOperator{\Res}{Res}
\DeclareMathOperator{\Char}{Char}

\DeclareMathOperator{\FinGpd}{FinGpd}
\DeclareMathOperator{\FinGfd}{FinGfd}
\DeclareMathOperator{\one}{(i)}
\DeclareMathOperator{\three}{(iii)}
\DeclareMathOperator{\FP}{FP}

\DeclareMathOperator{\two}{(ii)}
\DeclareMathOperator{\op}{op}
\DeclareMathOperator{\crll}{Corollary }

\DeclareMathOperator{\thrm}{Theorem }
\DeclareMathOperator{\mm}{max}
\theoremstyle{plain}
\newtheorem{theorem}{Theorem}[section]

\newtheorem{lemma}[theorem]{Lemma}
\newtheorem{proposition}[theorem]{Proposition}
\newtheorem{corollary}[theorem]{Corollary}
\newtheorem{definition}[theorem]{Definition}

\newtheorem{question}[theorem]{Question}

\theoremstyle{remark}
\newtheorem{remark}[theorem]{Remark}
\newtheorem{motivation}[theorem]{Motivation}
\newtheorem{note}[theorem]{Note}

\numberwithin{equation}{section}

\begin{document}

\title[Gorenstein analogues of a projectivity criterion over group algebras]
{Gorenstein analogues of a projectivity criterion over group algebras}

\author[Rudradip Biswas]
{Rudradip Biswas }
\address{Department of Mathematics\\
University of Warwick\\
Zeeman Building, Coventry CV4 7AL, UK}
\email{rudradip.biswas@warwick.ac.uk}

\author[Dimitra-Dionysia Stergiopoulou]
{Dimitra-Dionysia Stergiopoulou}
\address{Department of Mathematics\\
University of Athens\\ Athens 15784\\
Greece}
\email{dstergiop@math.uoa.gr}

%    General inf
\subjclass[2010]{Primary: 20C07, Secondary: 18G05, 20K40.}

\date{\today}

\keywords{}

\begin{abstract}
{We formulate and answer Gorenstein projective, flat, and injective analogues of a classical projectivity question for group rings under some mild additional assumptions. Although the original question, that was proposed by Jang-Hyun Jo in 2007, was for integral group rings, in this article, we deal with more general commutative base rings. We make use of the vast developments that have happened in the field of Gorenstein homological algebra over group rings in recent years, and we also improve and generalize several existing results from this area along the way.}
\end{abstract}

\maketitle

\section{Introduction and motivating questions}\label{s0}

In \cite{jo1}, Jang-Hyun Jo asked the following question:

\begin{question}\label{moti} Let $G$ be a group such that $\{\mathbb{Z}G$-projectives$\}=\{\mathbb{Z}G$-modules that are $\mathbb{Z}$-projective and have finite projective dimension over $\mathbb{Z}G\}$ (this property is known to be true for finite $G$). Then, does $G$ have to be finite?

\end{question}

Jo showed that the answer to this question is ``yes" if it is already assumed that $G$ is a group locally in Kropholler's hierarchy with the class of finite groups, $\mathscr{F}$, as the base class (in other words, $G \in \LH\mathscr{F}$ - see Definition \ref{hx-def}, this is a very large family of groups with only a handful of groups known to lie outside it). 

The first author revisited this question in \cite{b} and by using what was then some recent developments in the understanding of the behaviour of several cohomological invariants of groups, generalized Jo's result to groups in $\LH\mathscr{F}_{\phi,\mathbb{Z}}$ (with $\mathscr{F}_{\phi,\mathbb{Z}}$ denoting the class of all groups of type $\Phi$ over $\mathbb{Z}$ - a group $G$ is said to be of type $\Phi$ \cite{t} over $\mathbb{Z}$ if for every $\mathbb{Z}G$-module $M$, $\pd_{\mathbb{Z}G}(M)<\infty$ iff $\pd_{\mathbb{Z}H}(M)<\infty$ for all finite $H\leq G$) with a much shorter proof. 

In the last few years, a lot of exciting progress has occurred in the Gorenstein homological literature for group rings, especially with the commutative base rings required to only satisfy much milder finiteness conditions (like, the supremum over the projective/flat dimension of injectives being finite) than finite global dimension. So, it is reasonable to ask whether one can formulate and answer Gorenstein projective/flat/injective analogues of Question \ref{moti} for rings more general than $\mathbb{Z}$. For any ring $R$, $\GProj(R)$, $\GFlat(R)$, and $\GInj(R)$ denote the classes of (left) Gorenstein projective, Gorenstein flat, and Gorenstein injective $R$-modules separately.

\begin{question}\label{central} Take a commutative ring $A$ and a group $G$. Under mild assumptions on $A$ and $G$, reach completely group theoretic conclusions involving $G$ in the following three cases separately:

\begin{itemize} 

\item[(i)] when $\GProj(A) \bigcap \{AG$-modules with finite Gorenstein projective dimension$\}=\GProj(AG)$.

\item[(ii)] when $\GFlat(A) \bigcap \{AG$-modules with finite Gorenstein flat dimension$\}=\GFlat(AG)$.

\item[(iii)] when $\GInj(A) \bigcap \{AG$-modules with finite Gorenstein injective dimension$\}=\GInj(AG)$.
\end{itemize}

\end{question}

We approach Question \ref{central} in this paper, and our answers are summarized below. The assumptions that will be placed on the group $G$ will be in the language of Kropholler's hierarchy (see Definition \ref{hx-def}) - this hierarchy starts with a base class of groups which, in the hypotheses of Theorem \ref{summary}.(i)-(iii), are either the class of all groups admitting a characteristic module (see Definition \ref{char-wchar}; this class is denoted by $\mathscr{X}_{\Char}$) or the class of all groups admitting a \textit{weak} characteristic module (again, see Definition \ref{char-wchar}; this class is denoted by $\mathscr{X}_{\WChar}$); both of these classes are larger than the classes of groups taken as base classes in the previous approaches to Question \ref{moti} by Jo in \cite{jo1} and the first author in \cite{b}.

\begin{theorem}\label{summary} Let $A$ be a commutative ring and let $G$ be a group.  

\begin{itemize}

\item[(i)] $[\thrm \ref{m1}]$ Assume the hypothesis of $\ref{central}\one$. If, additionally, the supremum over the projective dimension of $A$-injectives is finite, and $G \in \LH\mathscr{X}_{\Char}$, then $G$ has to be finite. 

\item[(ii)] $[\thrm \ref{m2}]$ Assume the hypothesis of $\ref{central}\two$. If, additionally, $A$ is $\aleph_0$-Noetherian (i.e. all ideals are countably generated), the supremum over the flat dimension of $A$-injectives is finite, and $G \in \LH\mathscr{X}_{\WChar}$, then $G$ has to be locally finite.

\item[(iii)] $[\thrm \ref{m4}]$ Assume the hypothesis of $\ref{central}\three$. If, additionally, $A=\mathbb{Z}$, and $G\in\LH\mathscr{X}_{\Char}$, then $G$ has to be finite.

\end{itemize}
\end{theorem}

 In Theorems \ref{summary}.(i)-(ii), we do not require $A$ to be $\mathbb{Z}$ to reach our final group theoretic conclusion, and although in Theorem \ref{summary}.(iii), we do need $A$ to be $\mathbb{Z}$ for the final conclusion, the reader will see that the actual statement of Theorem \ref{m4} is for more general base rings (see Subsection \ref{s25}).

\textbf{Structure of this paper:} Theorem \ref{summary}.(i) is proved quite directly in Subsection \ref{s20}. Theorem \ref{summary}.(ii) is proved in Subsection \ref{s22} and Theorem \ref{summary}.(iii) in Subsection \ref{s24}, with Subsections \ref{s21} and \ref{s23} devoted to developing some technical machinery involving several (co)homological invariants that will turn out to be quite crucial for Subsections \ref{s22} and \ref{s24} respectively. Several of these intermediate results from \ref{s21} and \ref{s23} are interesting in their own right, we end this section highlighting one of them:

\begin{theorem} $[\crll \ref{ghd0char}]$\label{ghd0-summary} Let $A$ be a commutative ring satisfying the hypothesis on $A$ placed in $\thrm \ref{summary}\two$, and let $G$ be a group. Then, the Gorenstein homological dimension of $G$ with respect to $A$ is $0$ iff $G$ is locally finite. 

\end{theorem}

Although the Gorenstein cohomological version of Theorem \ref{ghd0-summary} has been known for a while, Theorem \ref{ghd0-summary} is not only new, but it improves on the previously existing versions of this result where $A$ had to be either $\mathbb{Z}$ or one had to require $\Ghd_A(G)=0$ for all commutative rings $A$.
\section{Background}\label{s1}

\subsection{Gorenstein homological background}

Since we will be dealing with Gorenstein projectives a fair amount in this paper, it is fitting to start with its definition. All of our modules will be left modules.

\begin{definition} \label{d1} Let $R$ be a ring.  
\begin{itemize}
\item[(a)] A Gorenstein projective $R$-module is an $R$-module that arises as a kernel in a totally acyclic complex of $R$-projectives. Recall that an acyclic complex of $R$-projectives, $\{P_i, d_i\}_{i \in \mathbb{Z}}$, is called totally acyclic if, for any $R$-projective $Q$, $\Hom_R(P_{*},Q)$ is acyclic. The class of all Gorenstein projective $R$-modules is denoted by $\GProj(R)$, and the class of all $R$-modules with finite Gorenstein projective dimension is denoted by $\GProj^{<\infty}(R)$.

Recall that, for any $R$-module $M$, the Gorenstein projective dimension, denoted $\Gpd_R(M)$, is the smallest ineteger $n$ such that there exists an exact sequence $$0 \rightarrow G_n \rightarrow G_{n-1} \rightarrow \cdots \rightarrow G_0 \rightarrow M \rightarrow 0,$$ with each $G_i \in \GProj(R)$; if no such $n$ exists, we say $\Gpd_R(M)$ is not finite. 

Dually, an $R$-module $N$ is called Gorenstein injective if it occurs as a kernel in a totally acyclic complex of $R$-injectives (i.e. an acyclic complex of $R$-injectives $I_*$ such that $\Hom_R(J,I_*)$ is also acyclic for any $R$-injective $J$). The class of Gorenstein injective $R$-modules is denoted by $\GInj(R)$, and Gorenstein injective dimension of $N$ as an $R$-module, denoted by $\Gid_R(N)$, is the smallest integer $n$ such that there is an exact sequence $$0 \rightarrow N \rightarrow G_0 \rightarrow G_1 \rightarrow ...\rightarrow G_n \rightarrow 0,$$ with each $G_i \in \GInj(R)$; if no such $n$ exists, we say $\Gid_R(N)$ is not finite. The class of all $R$-modules with finite Gorenstein injective dimension is denoted by $\GInj^{<\infty}(R)$.

When $R$ is a group algebra $AG$, for any discrete group $G$ and any commutative ring $A$, the Gorenstein cohomological dimension of $G$ over $A$, written $\Gcd_A(G)$, is defined as $\Gpd_{AG}(A)$. 

\item[(b)] A Gorenstein flat $R$-module is an $R$-module that arises as a kernel in an acyclic complex of flat $R$-modules $(F_i, \delta_i)_{i \in \mathbb{Z}}$ that has the property that for any injective right $R$-module $I$, the complex $I \otimes F_{*}$ is acyclic. The class of all Gorenstein flat $R$-modules is denoted by $\GFlat(R)$, and the class of all $R$-modules with finite Gorenstein flat dimension is denoted by $\GFlat^{<\infty}(R)$.

Recall that, for any $R$-module $M$, the Gorenstein flat dimension, denoted by $\Gfd_R(M)$, is the smallest ineteger $n$ such that there exists an exact sequence $$0 \rightarrow G_n \rightarrow G_{n-1} \rightarrow . . . \rightarrow G_0 \rightarrow M \rightarrow 0,$$ with each $G_i \in \GFlat(R)$; if no such $n$ exists, we say $\Gfd_R(M)$ is not finite. 

When $R$ is a group algebra $AG$, for any discrete group $G$ and any commutative ring $A$, the Gorenstein homological dimension of $G$ over $A$, written $\Ghd_A(G)$, is defined as $\Gfd_{AG}(A)$. 

\item[(c)] $\Findim(R)$ ($\fidim(R)$, $\ffdim(R)$, $\FinGpd(R)$, $\fgidim(R)$ and $\FinGfd(R)$, respectively) denotes the supremum over the projective dimension (respectively, injective dimension, flat dimension, Gorenstein projective dimension, Gorenstein injective dimension and Gorenstein flat dimension) of all (left) $R$-modules with finite projective dimension (respectively, finite injective dimension, finite flat dimension, finite Gorenstein projective dimension, finite Gorenstein injective dimension and finite Gorenstein flat dimension).

\item[(d)] $\spli(R)$ and $\silp(R)$ denote the supremum over the projective dimension of all $R$-injectives and the supremum over the injective dimension of all $R$-projectives, respectively - these invariants were first introduced in \cite{gg}. Similarly, $\sfli(R)$ and $\silf(R)$ denote the supremum over the flat dimension of all $R$-injectives and the supremum over the injective dimension of all $R$-flats, respectively. 
\end{itemize}
\end{definition}

\begin{remark}
    Almost all of the modules considered in this paper will be modules over self-opposite rings (either a commutative ring or a group ring over a commutative base). So, we will not need to worry about left/right modules. 
\end{remark}

Quite a lot of known results involving the invariants defined in Definition \ref{d1}(c)-(d) will be used in proving several technical results in this paper that will, in turn, play important roles in proving the main theorems. We collect some of them in one place below:

\begin{proposition}\label{finchar}  

\begin{itemize}
    \item[(i)] For any ring $R$, $\Findim(R)=\FinGpd(R)$ \cite[Theorem 2.28]{hh}, $\fidim(R) = \fgidim(R)$ \cite[Theorem 2.29]{hh}, and $\ffdim(R)=\FinGfd(R)$ \cite[Proposition 2.3]{em}.

    \item[(ii)] For any group $G$, $G$ is finite $\Longleftrightarrow \spli(\mathbb{Z}G)=1 \Longleftrightarrow \Gcd_A(G)=0$ for any commutative ring $A$ - this is a combination of \cite[Theorem 3]{dt} and \cite[Corollary 2.3]{et}.
    
\end{itemize}

\end{proposition}

There are several close, but not exact, analogues of Proposition \ref{finchar}.(ii) characterizing groups with Gorenstein homological dimension zero, but none that deals with any fixed commutative base ring of coefficients. We establish a result of this type in Corollary \ref{ghd0char} which was highlighted in Section \ref{s0}.

\subsection{Group theoretic background}

As mentioned in Section \ref{s0}, most of the assumptions that we will place on our groups will be in terms of their position in Kropholler's hierarchy. We give an explicit definition of this hierarchy below.

\begin{definition}\cite{krop94}\label{hx-def} Let $\mathscr{X}$ be a class of groups,

$\hh_0\mathscr{X}:=\mathscr{X}$.

Inductively, for any successor ordinal $\alpha$, $\hh_{\alpha}\mathscr{X}$ is defined in the following way - a group $G$ is said to be in $\hh_{\alpha}\mathscr{X}$ iff there exists a finite-dimensional contractible CW-complex on which $G$ acts with stabilizers in $\hh_{\alpha-1}\mathscr{X}$.

For any ordinal $\alpha$, $\hh_{<\alpha}\mathscr{X}:=\bigcup_{\beta<\alpha}\hh_{\beta}\mathscr{X}$. When $\alpha$ is a limit ordinal, $\hh_{\alpha}\mathscr{X}:= \hh_{<\alpha}\mathscr{X}$. 

A group $G$ is said to be in $\hh\mathscr{X}$ if $G \in \hh_{\alpha}\mathscr{X}$ for some ordinal $\alpha$, and it is said to be in $\LH\mathscr{X}$ iff every finitely generated subgroup of it is in $\hh\mathscr{X}$. 
    
\end{definition}

\begin{remark}\textbf{General facts, egs, and non-egs about groups in the hierarchy:} It is easy to see that, for any fixed $\mathscr{X}$, $\hh_{\alpha}\mathscr{X}\subseteq \hh_{\beta}\mathscr{X}$, if $\alpha<\beta$. In fact, if $\mathscr{X}=\mathscr{F}$, the class of all finite groups, it was shown in \cite{jkl} that $\hh_{<\alpha}\mathscr{F}$ is strictly smaller than $\hh_{\alpha}\mathscr{F}$ for all ordinals $\alpha\leq \omega_1$, the first uncountable ordinal. $\LH\mathscr{F}$ and $\hh\mathscr{F}$ are very large classes of groups. $\LH\mathscr{F}$ contains all soluble-by-finite groups, all linear groups, all groups with finite-dimensional models for their classifying space of proper actions, and so on. Very few groups are known to lie outside $\LH\mathscr{F}$ or $\hh\mathscr{F}$ - some notable examples being Thompson's group $F$ \cite{krop94} or the first Grigorchuk group \cite{gandini}. 
    
\end{remark}

\begin{motivation}\label{char-history} \textbf{(The module $B(G,A)$ and the background behind characteristic and weak characteristic modules)} For any group $G$ and any commutative ring $A$, it is well-known that we have an $AG$-module, written $B(G,A)$, formed of those functions $G \rightarrow A$ with finite range, and it is also known \cite{ben1,ck97} that $B(G,A)$ is $A$-free and free over $AH$ for any finite subgroup $H\leq G$, and that embedding the constant functions, we get an $A$-split $AG$-monomorphism $A \hookrightarrow B(G,A)$. If $G$ is taken to be a type $\Phi$ group over $A$ (recall from Section \ref{s0} that these are those groups where any $AG$-module has finite projective dimension iff its restriction to every finite subgroup has finite projective dimension), then it is obvious that $\pd_{AG}B(G,A)<\infty$. Over the last three decades, the module $B(G,A)$ has played a very central role in developments in the understanding of cohomological properties of infinite groups or the closed model category structure on modules over group algebras - one can, for example, look at the papers that have used the machinery developed in \cite{ben1}, \cite{ck97}, and other papers.
\end{motivation}

Motivated by the facts summarised in Motivation \ref{char-history}, Talelli \cite{tal-char} introduced (over $\mathbb{Z}$) the concept of a characteristic module for a group, and on a similar vein, the second author introduced the notion of a weak characteristic module for a group in \cite{dds2}. 

\begin{definition}\label{char-wchar} Let $A$ be a commutative ring and let $G$ be a group.
\begin{itemize}
    \item[(i)] $G$ is said to admit a characteristic module over $A$ (when the $A$ is clear from context, we omit ``over $A$") if there is an $AG$-module $T$ such that $\first$ $T$ is $A$-free, $\ditio$ there is an $A$-split $AG$-monomorphism $A \rightarrow T$, and $\tritio$ $\pd_{AG}(T)<\infty$. With $A$ fixed and clear from the context, the class of all groups admitting a characteristic module is denoted $\mathscr{X}_{\Char}$.

    \item[(ii)] $G$ is said to admit a weak characteristic module over $A$ (again, we omit ``over $A$" if it is clear from the context) if there is an $AG$-module $S$ such that $\first$ $S$ is $A$-flat, $\ditio$ there is an $A$-pure $AG$-monomorphism $A \rightarrow S$, and $\tritio$ $\flatdim_{AG}(S)<\infty$. With $A$ fixed and clear from the context, the class of all groups admitting a weak characteristic module is denoted $\mathscr{X}_{\WChar}$. 
\end{itemize} 
\end{definition}

 Taking $\mathscr{X}_{\Char}$ as the base class in Kropholler's hierarchy, much has been studied about the behaviour of (co)homological invariants and related questions for groups in $\LH\mathscr{X}_{\Char}$ in recent papers like \cite{dds,dds2}. And, taking $\mathscr{X}_{\WChar}$ as the base class in Kropholler's hierarchy, a lot has been studied about the behaviour of (co)homological invariants and related questions for groups in $\LH\mathscr{X}_{\WChar}$ by the second author in \cite{dds2}. Clearly, $\{$Finite groups$\}\subset \{$Type $\Phi$ groups$\}\subseteq \mathscr{X}_{\Char}\subseteq \mathscr{X}_{\WChar}$, keeping a commutative ring $A$ fixed.

\section{New results}\label{s2}

\subsection{Gorenstein projective version}\label{s20}

The main goal of this subsection is to prove the result highlighted in Theorem \ref{summary}.(i):

\begin{theorem}\label{m1} Let $A$ be a commutative ring such that $\spli(A)<\infty$. Let $G$ be a group such that $\GProj(A) \bigcap \GProj^{<\infty}(AG) \subseteq \GProj(AG)$. Now, assume that $G \in \LH\mathscr{X}_{\Char}$, then $G$ has to be finite.

\end{theorem}
First, we record a small lemma without using the ``$G \in \LH\mathscr{X}_{\Char}$" assumption that will be useful in proving Theorem \ref{m1}.

\begin{lemma}\label{m01} Let $A$ be a commutative ring such that $\spli(A)<\infty$, and let $G$ be a group such that $\GProj(A) \bigcap \GProj^{<\infty}(AG) \subseteq \GProj(AG)$. Then, $\Findim(AG) \leq \spli(A)$.
\end{lemma}

\begin{proof} Let $\spli(A)=n$. Note that the supremum over the Gorenstein projective dimension of all $A$-modules is exactly $n$ - this is because, in this case, $\silp(A)\leq n$ as $A$ is commutative (this can be checked directly but, for a reference, see Corollary $24$ of \cite{de}), and then we can just invoke Theorem 2.2.$(\gamma)$ of \cite{br}. Now, let $M$ be an $AG$-module such that $\Gpd_{AG}(M)<\infty$. Let $P_{*}$ be an $AG$-projective resolution of $M$ where $K_n$ is the $n$-th kernel, so we have an exact sequence $$0 \rightarrow K_n \rightarrow P_{n-1} \rightarrow \cdots\rightarrow P_0 \rightarrow M \rightarrow 0.$$ Here, $K_n \in \GProj^{<\infty}(AG)$. Also, at the same time, since $\Gpd_A(M)\leq n$ and each $P_i \in \GProj(A)$ due to being $A$-projective, $K_n \in \GProj(A)$. This gives us $K_n \in \GProj(A) \bigcap \GProj^{<\infty}(AG) \subseteq \GProj(AG)$. So, $\Gpd_{AG}(M)\leq n$. Thus, $\FinGpd(AG)\leq n$. But, recall from Proposition \ref{finchar}.(i) that $\Findim(AG)=\FinGpd(AG)$, so $\Findim(AG)\leq n$.\end{proof}

\begin{proof}[Proof of Theorem \ref{m1}] Lemma \ref{m01} tells us that $\Findim(AG)<\infty$. Since $G \in \LH\mathscr{X}_{\Char}$ and $\spli(A)<\infty$, invoking \cite[Remark 6.9.(ii)]{dds2} we obtain that $\spli(AG)=\Findim(AG)<\infty$. Now, again as $\spli(A)<\infty$, we have $\spli(AG)<\infty$ (i.e. $\silp(AG)=\spli(AG)<\infty$ \cite[Corollary 24]{de}) $\Longleftrightarrow \Gcd_A(G) (= \Gpd_{AG}(A))<\infty$ (this equivalence is due to \cite[Theorem 2.14]{dds}).  Clearly, $A \in \GProj(A)$, so we have that $A \in \GProj(A) \bigcap \GProj^{<\infty}(AG)$, and now by the hypothesis, we have that $A \in \GProj(AG)$, i.e. $\Gcd_A(G)=0$. Thus, we are done by Proposition \ref{finchar}.(ii).  \end{proof}

Before proceeding towards the flat version of Theorem \ref{m1}, we need to take a small detour and prove a result regarding the Gorenstein homological dimension of groups.

\subsection{What does $\Ghd_A(G)=0$ imply?}\label{s21}
We start with some results from a recent paper by Kaperonis and the second author \cite{ks} that will be of use to us in this subsection. 

\begin{lemma}\label{lemma-ks} Let $A$ be a commutative ring and let $G$ be a group. Then,
\begin{itemize}
    \item[(i)]If $\sfli(A)<\infty$, $\Ghd_A(H) \leq \Ghd_A(G)$, for any subgroup $H \leq G$ - see \cite[Proposition 4.1]{ry} or \cite[Proposition 5.11]{ks}.
    \item[(ii)]$\mm \{\Ghd_A(G),\sfli(A)\} \leq \sfli(AG) \leq \Ghd_A(G)+\sfli(A)$ \cite[Corollary 6.2]{ks}.
\end{itemize}
\end{lemma}

The results collected in Lemma \ref{lem2} are quite derivative of some results from \cite{et2011, de} but not all of them have appeared in the form below.
 
\begin{lemma}\label{lem2} Let $A$ be a commutative ring and let $G$ be a group such that the group ring $AG$ is $\aleph_0$-Noetherian. Then,
\begin{itemize}
\item[(i)] $\silp(AG)=\spli(AG)$ - this follows from \cite[Corollary 27]{de}.
\item[(ii)] $\sfli(AG)\leq \silp(AG)=\spli(AG)\leq \sfli(AG)+1\leq \Ghd_A(G)+\sfli(A)+1$. 
\item[(iii)]$\id_{AG}(AG)\leq \Ghd_A(G)+\sfli(A)+1.$
\end{itemize}
\end{lemma}

\begin{proof}

(ii) As projectives are flat, clearly $\sfli(AG)\leq \spli(AG)$. The ``$\silp(AG)=\spli(AG)$" equality is given by part (i) of this lemma. Now, to show that $\spli(AG)\leq \sfli(AG)+1$, we can assume that $\sfli(AG)<\infty$. As $AG \cong (AG)^{\op}$, we also get that $\sfli((AG)^{\op})<\infty$, and this lets us use \cite[Theorem 2.9]{et2011} to conclude that $\spli(AG)\leq \sfli(AG)+1$. The last inequality, i.e. the claim that $\sfli(AG)+1 \leq \Ghd_A(G)+\sfli(A)+1$, follows from Lemma \ref{lemma-ks}.(ii).

(iii) This is a direct corollary of part (ii) of this lemma as clearly, $\id_{AG}(AG)\leq \silp(AG)$.
\end{proof}

\begin{remark} In the proof of the upcoming Theorem \ref{ghd0}, only part (iii) of the three statements of Lemma \ref{lem2} will be used but, as the reader can notice from the proof of Lemma \ref{lem2}, part (iii) cannot be established without using both parts (i) and (ii).

\end{remark}

Using Lemmas \ref{lemma-ks} and \ref{lem2}, we can achieve a generalization of a 2012 theorem of Emmanouil's \cite[Proposition 2.1]{ejoa}.

\begin{theorem}\label{ghd0} Let $A$ be an $\aleph_0$-Noetherian commutative ring satisfying $\sfli(A)<\infty$ and let $G$ be a finitely generated group. Then, the following conditions are equivalent.
\begin{itemize}
    \item[(i)]$G$ is finite.
    \item[(ii)]$\Ghd_A(G)=0$.
    \item[(iii)]$\hh_i(G,\rule{0.3cm}{0.15mm})$ vanishes on injective $AG$-modules $\forall i\geq 1$, and $\id_{AG}(AG)\leq \sfli(A)+1$.
    \item[(iv)]$\hh_i(G,\rule{0.3cm}{0.15mm})$ vanishes on injective $AG$-modules $\forall i\geq 1$, and $\id_{AG}(AG)< \infty$.
    \item[(v)] $\hh_i(G,\rule{0.3cm}{0.15mm})$ vanishes on injective $AG$-modules $\forall i\geq 1$, and $\exists$ a non-zero induced $AG$-module of finite injective dimension.
\end{itemize}
\end{theorem} 

\begin{proof} The ``(i) $\Rightarrow$ (ii)" direction is true for any commutative ring $A$ and without any finite generation condition on $G$ - this can be seen quite directly but for a reference, see \cite[Corollary 5.2]{dds}. 

(ii) $\Rightarrow$ (iii): As $A$ is $\aleph_0$-Noetherian and $G$ is finitely generated, $AG$ is $\aleph_0$-Noetherian. We have $\Ghd_A(G)=0$ by assumption, so we can use Lemma \ref{lem2}.(iii) to deduce that $\id_{AG}(AG)\leq \sfli(A)+1$. By the definition of the Gorenstein homological dimension, we have that $\Gfd_{AG}(A)=0$. Since every ring is GF-closed (see \cite[Corollary 4.12]{ss}), we can invoke \cite[Theorem 2.8]{bennis} to get that $\tor^{AG}_i(I,A)=0$ for any $AG$-injective $I$ and for all $i\geq 1$. Now, $\tor^{AG}_i(I,A) \cong \tor^{AG}_i(A,I)$ $\forall i$, as $AG \cong (AG)^{\op}$. So, $\tor^{AG}_i(A,I)=0$ for any $AG$-injective $I$ and for all $i\geq 1$. This precisely means that $\hh_i(G,\rule{0.3cm}{0.15mm})$ vanishes on injective $AG$-modules $\forall i\geq 1$.

(iii) $\Rightarrow$ (iv) is trivial.

In \cite[Proposition 2.1]{ejoa}, Emmanouil works with $A=\mathbb{Z}$ and has a different statement in place of our (ii). But his proof of (iv) $\Rightarrow$ (v) goes through in our case verbartim. The same is true with the proof of the ``(v) $\Rightarrow$ (i)" direction but since (i) is a group-theoretic statement, it is worth explaining why Emmanouil's proof (i.e. the proof for $A=\mathbb{Z}$) works more generally. Let us assume that (v) holds but $G$ is infinite. First, following Emmanouil's proof, one gets an injective $AG$-module $I$ satisfying $I_G\neq 0$. It is a direct application of a result of Strebel's (in Strebel's original paper, it is clearly stated that $A$ can be taken to be any commutative ring \cite[Section 2 and Section 4]{strebel}) that, as $G$ is finitely generated, the group of $G$-coinvariants of any coinduced $AG$-module is trivial. From this fact, it is an easy deduction (as observed by Ikenaga in \cite[Corollary 1.6]{ik} for $A=\mathbb{Z}$, but one can replace $\mathbb{Z}$ with any commutative ring here) that if $G$ is allowed to be both infinite and finitely generated, then the group of $G$-coinvariants of any injective $AG$-module is trivial. Thus, we have a contradiction, and therefore, $G$ must be finite if (v) is satisfied.   \end{proof}

\begin{corollary}\label{ghd0char} Let $A$ be an $\aleph_0$-Noetherian commutative ring satisfying $\sfli(A)<\infty$, and let $G$ be a group. Then, $\Ghd_A(G)=0 \Leftrightarrow G$ is locally finite.

\end{corollary}

\begin{proof} Let $\Ghd_A(G)=0$. Take $H$ to be a finitely generated subgroup of $G$ - we need to show that $H$ is finite. Lemma \ref{lemma-ks}.(i) tells us that $\Ghd_A(H)=0$, and then the ``(ii) $\Rightarrow$ (i)" part of Theorem \ref{ghd0} forces $H$ to be finite.

For the other direction, we don't need the ``$\aleph_0$-Noetherian" condition on $A$. For any commutative ring $A$ and any locally finite $G$, $\Ghd_A(G)=0$ - see \cite[Remark 3.6]{dds}.\end{proof} 

\subsection{Gorenstein flat version of Theorem \ref{m1}}\label{s22}

The main objective of this subsection is proving the following result. If we draw comparisons with Theorem \ref{m1}, we will see that other than the obvious changes with ``Gorenstein flats" replacing ``Gorenstein projectives", the starting hypothesis on $A$ is the finiteness of $\sfli(A)$ as opposed to the finiteness of $\spli(A)$ (it turns out that when $A$ is Noetherian, $\spli(A)=\sfli(A)$ - see Corollary \ref{spli-sfli-noetherian}.(i)), the class of groups can now be changed from $\LH\mathscr{X}_{\Char}$ to $\LH\mathscr{X}_{\WChar}$, and the finiteness conclusion on the group changes to local finiteness.

\begin{theorem}\label{m2} Let $A$ be a commutative ring such that $\sfli(A)<\infty$. Let $G$ be a group such that $\GFlat(A) \bigcap \GFlat^{<\infty}(AG) \subseteq \GFlat(AG)$. Now, assume that $G \in \LH\mathscr{X}_{\WChar}$, then $\Ghd_A(G)=0$, and if, additionally, $A$ is $\aleph_0$-Noetherian, then $G$ is locally finite. 

\end{theorem} 
Our analog of Lemma \ref{m01} is the following technical lemma where, as expected, the ``$G \in \LH\mathscr{X}_{\WChar}$" assumption is not required.

\begin{lemma}\label{m02}  Let $A$ be a commutative ring such that $\sfli(A)<\infty$. Let $G$ be a group such that $\GFlat(A) \bigcap \GFlat^{<\infty}(AG) \subseteq \GFlat(AG)$. Then, $\ffdim(AG)\leq \sfli(A)$.

\end{lemma}

\begin{proof} Let $\sfli(A)=n$ - this forces the supremum over the Gorenstein flat dimension of all $A$-modules to be $n$ again using \cite[Theorem $2.4$]{cet21}. Let $M$ be an $AG$-module such that $\Gfd_{AG}(M)<\infty$, and let $F_{*}$ be an $AG$-flat resolution of $M$ where $K_n$ is the $n$-th kernel. This gives us an exact sequence $$0 \rightarrow K_n \rightarrow F_{n-1} \rightarrow \cdots \rightarrow F_1\rightarrow F_0 \rightarrow M \rightarrow 0,$$ with $K_n \in \GFlat^{<\infty}(AG).$ As, $\Gfd_A(M)\leq n$, and as each $F_i \in \GFlat(A)$ due to being $A$-flat, we have $K_n \in \GFlat(A)$ by \cite[Theorem $2.8$]{bennis} and \cite[Corollary $4.12$]{ss}. This gives us that $K_n \in \GFlat(A) \bigcap \GFlat^{<\infty}(AG) \subseteq \GFlat(AG)$, so $\Gfd_{AG}(M)\leq n$. Thus, $\FinGfd(AG)\leq n$. We are now done since $\FinGfd(AG)=\ffdim(AG)$ by Proposition \ref{finchar}.(i).\end{proof}

\begin{proof}[Proof of Theorem \ref{m2}] We get from Lemma \ref{m02} that $\ffdim(AG)<\infty$. Like in the proof of Theorem \ref{m1}, we note that as $G \in \LH\mathscr{X}_{\WChar}$ and $\sfli(A)<\infty$, \cite[Remark 6.9.(i)]{dds2} yields $\sfli(AG)=\ffdim(AG)<\infty$. Again, as $\sfli(A)<\infty$, we can use Lemma \ref{lemma-ks}.(ii) to see that $\sfli(AG)<\infty \Longleftrightarrow \Ghd_A(G)<\infty$. Thus, $A \in \GFlat^{<\infty}(AG)$, and as we always have $A \in \GFlat(A)$, we have $A \in \GFlat(A) \bigcap \GFlat^{<\infty}(AG) \subseteq \GFlat(AG)$, which means $\Ghd_A(G)=0$. Now, if $A$ is $\aleph_0$-Noetherian, we are done by Corollary \ref{ghd0char}.\end{proof}

\subsection{A useful equality of four invariants}\label{s23}
We will now move towards a Gorenstein injective version of Theorem \ref{m1}, but first, we need to make some progress in the understanding of a few of our invariants which we do in this subsection - one of the main achievements here is Theorem \ref{m3}.
\begin{note}
    Whenever we will take an arbitrary class of groups in this subsection, usually denoted by $\mathscr{X}$, it will be assumed that all finite groups are contained in $\mathscr{X}$.
\end{note}
\begin{definition}(In \cite{ck98}, an analogous notion for ``projective dimension" with $\mathscr{X}=\{$all finite groups$\}$ was introduced.) For any group $G$, any commutative ring $A$, and any class of groups $\mathscr{X}$, we define $\kappa_{\Ij}(AG;\mathscr{X}):=\supr\{\id_{AG}(M): M$ is any $AG$-module satisfying $\id_{AH}(M)<\infty, \forall$ subgroup $H\leq G, H \in \mathscr{X}  \}$
\end{definition}

Parts (i) and (ii) of Lemma \ref{lemb} below are quite standard, and while proving (iii) does require using a deep result from a recent paper by the second author \cite{dds}, the rest of the proof of (iii) is quite simple.

\begin{lemma}\label{lemb}\begin{itemize}

\item[(i)] For any ring $R$, $\fidim(R) \leq \spli(R)$.

\item[(ii)] For any group $G$, any commutative ring $A$, and any subgroup $H \leq G$, $\fidim(AH)\leq \fidim(AG)$.

\item[(iii)] Let $G$ be a group and let $A$ be a commutative ring such that $\spli(A)<\infty$. Then, $\silp(AG)\leq \kappa_{\Ij}(AG;\mathscr{X}_{\Char})$.
\end{itemize}

\end{lemma}

\begin{proof} 
(i) Let $N$ be an $R$-module such that $\id_R(N)<\infty$, and let $\id_R(N)=n$. So, $\Ext^i_R(M,N)=0$ for any $R$-module $M$ and any $i\geq n+1$, and there must exist an $R$-module $L$ such that $\Ext^n_R(L,N)\neq 0$. Take a truncated injective resolution of $L$, $0 \rightarrow L \rightarrow I \rightarrow K \rightarrow 0$, where $I$ is an injective $R$-module. Applying $\Ext^{*}_R(\rule{0.3cm}{0.15mm},N)$ to this short exact sequence, we get the long exact sequence $$\cdots \rightarrow \Ext^n_R(K,N) \rightarrow \Ext^n_R(I,N) \rightarrow \Ext^n_R(L,N) \rightarrow \Ext^{n+1}_R(K,N) \rightarrow \cdots.$$ We know that $\Ext^{n+1}_R(K,N)=0$, so if $\Ext^n_R(I,N)=0$, we will have that $\Ext^n_R(L,N)=0$ which is not possible. So, $\Ext^n_R(I,N)\neq 0$. So, $\pd_R(I)\geq n$. Thus, $\spli(R)\geq n$. 

(ii) Let $M$ be an $AH$-module with $\id_{AH}(M)<\infty$. Therefore, as $\Coind^G_H(\rule{0.3cm}{0.15mm})$ takes an $AH$-injective resolution of $M$ to an $AG$-injective resolution of $\Coind^G_HM$ of the same length, $\id_{AG}(\Coind^G_HM)\leq \id_{AH}(M)$. Again, as $M$, as an $AH$-module, is a summand of $\Res^G_H\Coind^G_HM$, we have $$\id_{AH}(M)\leq \id_{AH}(\Res^G_H\Coind^G_HM)\leq \id_{AG}(\Coind^G_HM)\leq\id_{AH}(M)$$ So, $\id_{AH}(M)=\id_{AG}(\Coind^G_HM)\leq \fidim(AG)$, and we are done.

(iii) Let $H \leq G$ such that $H \in \mathscr{X}_{\Char}$. As $\spli(A)<\infty$, \cite[Theorem 5.3]{dds} tells us that $\spli(AH)=\silp(AH)<\infty$. Let $P$ be a projective $AG$-module. Considering $P$ as a projective $AH$-module under the usual restriction functor, $\id_{AH}(P)\leq \silp(AH)<\infty$. By the definition of $\kappa_{\Ij}(AG;\mathscr{X}_{\Char})$, $\id_{AG}(P)\leq \kappa_{\Ij}(AG;\mathscr{X}_{\Char})$. Thus, $\silp(AG)\leq \kappa_{\Ij}(AG;\mathscr{X}_{\Char})$. 
\end{proof}
   
The following result can be looked at as a dual version of \cite[Lemma 3.11]{b1}. Unlike Lemma \ref{lema} which deals with injective dimensions, \cite[Lemma 3.11]{b1} dealt with projective dimensions. A key input that has been added to the literature since the paper \cite{b1} was written is a dual (injective) analogue \cite[Lemma 5.7]{et25} of an old result of Benson's \cite[Lemma 5.6]{ben1} that plays a crucial role in the proof of Lemma \ref{lema} (the Benson result played the same role in the proof of \cite[Lemma 3.11]{b1}); we also use a transfinite induction trick that is quite standard in these contexts.

\begin{lemma}\label{lema} Let $A$ be a commutative ring and let $\mathscr{X}$ be a class of groups. Then, for any group $G \in \LH\mathscr{X}$, $\kappa_{\Ij}(AG;\mathscr{X})\leq \fidim(AG)$.

\end{lemma}

\begin{proof} Let $\fidim(AG)=n<\infty$, and let $M$ be an $AG$-module such that $\id_{AH}(M)<\infty$ for all $H \leq G$, $H\in \mathscr{X}$. 

\textbf{Claim A:} For any $H\leq G$ with $H \in \hh\mathscr{X}$, $\id_{AH}(M)\leq n$.

\textbf{Proof of Claim A:} We will proceed by transfinite induction on the smallest ordinal $\alpha$ such that $H\in H_{\alpha}\mathscr{X}$. First, we check the base case - if $\alpha=0$, then $H \in \mathscr{X}$, and $\id_{AH}(M)<\infty$, and as $\fidim(AH)\leq \fidim(AG)$ by Lemma \ref{lemb}.(ii), we have $\id_{AH}(M)\leq n$. Now, assume that Claim A is true for all $H\in H_{<\alpha}\mathscr{X}$. For the induction step, assume that $H\in H_{\alpha}\mathscr{X}$. From the definition of $H_{\alpha}\mathscr{X}$, it follows that $H$ acts on a finite-dimensional contractible CW-complex with stabilizers in $H_{<\alpha}\mathscr{X}$, and therefore there is a finite-length exact sequence of $AH$-modules, \begin{equation}\label{eq31} 0 \rightarrow C_r \rightarrow \cdots \rightarrow C_1 \rightarrow C_0 \rightarrow A \rightarrow 0,\end{equation} where each $C_i$ is a (not necessarily finite) direct sum of permutation modules of the form $A[H/H']$ with $H'\in H_{<\alpha}\mathscr{X}$, so for each $i$, let us write $C_i = \bigoplus_{j_i} A[H/H_{j_i}]$. As the exact sequence (\ref{eq31}) is $A$-split, we can apply $\Hom_A(\rule{0.3cm}{0.15mm},M)$ to it to obtain an exact sequence of $AH$-modules, 
\begin{equation}\label{eq32} 0 \rightarrow M \rightarrow \Hom_A(C_0,M)\rightarrow \Hom_A(C_1,M)\rightarrow \cdots \rightarrow \Hom_A(C_r,M)\rightarrow 0.\end{equation} 
Now, $\Hom_A(C_i,M)\cong \prod_{j_i}\Hom_A(A[H/H_{j_i}],M)\cong \prod_{j_i}\Coind^H_{H_{j_i}}M$. As noted in the proof of Lemam \ref{lemb}.(ii), $\id_{AH}(\Coind^H_{H_{j_i}}M)\leq \id_{AH_{j_i}}(M) \leq n$ (the last inequality follows from our induction hypothesis as $H_{j_i}\in H_{<\alpha}\mathscr{X}$). So, except $M$, all of the terms in the exact sequence (\ref{eq32}) have finite injective dimension as $AH$-modules. Thus, $\id_{AH}(M)<\infty$, i.e. $\id_{AH}(M)\leq \fidim(AH)\leq \fidim(AG)= n$. This ends the proof of Claim A.

\textbf{Claim B:} If $H\in \LH\mathscr{X}$, then $\id_{AH}(M)\leq n$. 

\textbf{Proof of Claim B:} If $H$ is countable, then $H\in H\mathscr{X}$ because $H$ acts on a tree with finitely generated vertex and edge stabilizers \cite[Lemma 2.5]{jkl}, and we are done by Claim A. If $H$ is uncountable, we proceed with the following induction hypothesis - assume that Claim B is true for all $H'< H$ of cardinality strictly smaller than that of $H$. As $H$ is uncountable, we can write $H = \bigcup_{\lambda<\delta}H_{\lambda}$, for some ordinal $\delta$, with $\id_{AH_{\lambda}}(M)\leq n$ for all $\lambda<\delta$. So, if we take $K_n$ to be the $n$-th cosyzygy in an $AG$-injective resolution of $M$, then $\id_{AH_{\lambda}}(K_n)=0$ for all $\lambda<\delta$. Now, \cite[Lemma 5.7]{et25} tells us that $\id_{AH}(K_n)\leq 2$ and consequently $\id_{AH}(M)<\infty$, i.e. $\id_{AH}(M)\leq \fidim(AH)\leq \fidim(AG)=n$ (although \cite[Lemma 5.7]{et25} is written for $A=\mathbb{Z}$, their proof works for any commutative $A$, which is not surprising because \cite[Lemma 5.7]{et25}, the injective analogue of \cite[Lemma 5.6]{ben1}, was also for all commutative base rings). This ends the proof of Claim B.

If we take $H=G$ in the statement of Claim B, we are done with the proof of the whole lemma. \end{proof}

\begin{theorem}\label{m3} Let $A$ be a commutative $\aleph_0$-Noetherian ring such that $\spli(A)<\infty$, and let $G$ be a group in $\LH\mathscr{X}_{\Char}$. Then, $$\fidim(AG)=\silp(AG)=\spli(AG)=\kappa_{\Ij}(AG;\mathscr{X}_{\Char}).$$
\end{theorem}

\begin{proof} As $A$ is $\aleph_0$-Noetherian, we have $\silp(AG)=\spli(AG)$ (see \cite[Proposition 5.6]{ks} or \cite[Remark 25.(ii)]{de}) - combining this fact with Lemma \ref{lemb}.(i), Lemma \ref{lemb}.(iii), and Lemma \ref{lema}, we have $\fidim(AG)\leq \spli(AG)=\silp(AG)\leq \kappa_{\Ij}(AG;\mathscr{X}_{\Char})\leq \fidim(AG)$. Thus, we are done.\end{proof}

\begin{remark} The ``projective" analogue of Theorem \ref{m3} was proved in \cite[Theorem 3.1]{b1} under a stronger hypothesis on $G$ with $A$ taken to be any commutative ring of finite global dimension, but Theorem \ref{m3} cannot be directly derived from that result. \cite[Theorem 3.1]{b1} was itself a generalization of an old theorem due to Cornick and Kropholler \cite[Theorem C]{ck98}.

\end{remark}

\subsection{Gorenstein injective version of Theorem \ref{m1}}\label{s24}

 The main result of this subsection is the following. One can draw obvious analogies with the starting hypotheses of Theorems \ref{m1} and \ref{m2}. The purpose of the ``$\aleph_0$-Noetherian" hypothesis on $A$ is to allow us to use Theorem \ref{m3}.

\begin{theorem}\label{m4} Let $A$ be a commutative $\aleph_0$-Noetherian ring such that $\spli(A)<\infty$, and let $G$ be a group such that $\GInj(A) \bigcap \GInj^{<\infty}(AG) \subseteq \GInj(AG)$. Assume that $G \in \LH\mathscr{X}_{\Char}$, then $\spli(AG)=\spli(A)$. In particular, if $A=\mathbb{Z}$, then $G$ has to be finite.

\end{theorem}

As in Subsections \ref{s20} and \ref{s22}, before proving Theorem \ref{m4}, we first record a technical lemma without the ``$G \in \LH\mathscr{X}_{\Char}$" assumption and the $\aleph_0$-Noetherian hypothesis on $A$.

\begin{lemma}(analogous to Lemmas \ref{m01} and \ref{m02})\label{lemd} Let $A$ be a commutative ring such that $\spli(A)<\infty$, and let $G$ be a group such that $\GInj(A) \bigcap \GInj^{<\infty}(AG) \subseteq \GInj(AG)$. Then, $\fidim(AG)\leq \spli(A)$.

\end{lemma}

\begin{proof} Let $\spli(A)=n<\infty$. As $A$ is commutative, $\silp(A)$ is also finite \cite[Corollary 24]{de}, and like in the proof of Lemma \ref{m01}, \cite[Theorem 2.2.$(\gamma)$]{br} tells us that the supremum of the Gorenstein injective dimension of all $A$-modules is also $n$. Let $M \in \GInj^{<\infty}(AG)$, and take a truncated $AG$-injective resolution of $M$, $$0 \rightarrow M \rightarrow I_0 \rightarrow I_1 \rightarrow \cdots \rightarrow I_{n-1} \rightarrow K_n \rightarrow 0.$$ As each $I_i\in \GInj(AG)$, $K_n \in \GInj^{<\infty}(AG)$. Also, as $\Gid_A(M)\leq n$, $K_n\in \GInj(A)$ \cite[Theorem 2.22]{hh}. So, by our hypothesis, $K_n \in \GInj(AG)$, which forces $\Gid_{AG}(M)\leq n$ (again using \cite[Theorem 2.22]{hh}). Thus, $\fgidim(AG)\leq \spli(A)$, and we are done as $\fidim(AG)=\fgidim(AG)$ by Proposition \ref{finchar}.(i).\end{proof}

\begin{proof}[Proof of Theorem \ref{m4}] Recall that, without any conditions on $A$ and $G$, $\spli(A)\leq \spli(AG)$. Now, combining Theorem \ref{m3} and Lemma \ref{lemd}, we get $\spli(A)\leq \spli(AG) = \fidim(AG) \leq \spli(A)$, and we are done.

If $A=\mathbb{Z}$, then $\spli(\mathbb{Z}G)=\spli(\mathbb{Z})=1$, and so $G$ is finite by Proposition \ref{finchar}.(ii). \end{proof}

\subsection{Do we really need $A=\mathbb{Z}$ for the group-theoretic conclusion in Theorem \ref{m4}?}\label{s25}

The following are two natural questions to ask looking at the ``$\spli(AG)=\spli(A)$" conclusion in the statement of Theorem \ref{m4}.

\begin{question}\label{non-z} Let $A$ be a commutative ring and let $G$ be a group. Can we obtain any group-theoretic conclusions in the following cases?
\begin{itemize}
	\item[(i)] $\spli(AG)=\spli(A)$,
	\item[(ii)]$\sfli(AG)=\sfli(A)$.
\end{itemize}
\end{question}

\begin{remark}
    When $A=\mathbb{Z}$ in Question \ref{non-z}, we know that $\spli(\mathbb{Z}G)=1 \Leftrightarrow G$ is finite (Proposition \ref{finchar}.(ii)) and $\sfli(\mathbb{Z}G)=1\Leftrightarrow G$ is locally finite (Corollary \ref{spli-sfli-noetherian}.(ii)).
\end{remark}

Keeping Question \ref{non-z} in mind, we make the following definitions.

\begin{definition}
    For any commutative ring $A$, we say $A$ satisfies the $\splif$ principle if, for any group $G$, whenever $\spli(AG)=\spli(A)$, $G$ is finite. Similarly, we say $A$ satisfies the $\sflilf$ (resp. $\splilf$) principle if, for any group $G$, whenever $\sfli(AG)=\sfli(A)$ ($\spli(AG)=\spli(A)$, respectively), $G$ is locally finite.
\end{definition}
 The main message of this subsection is the following observation. 
\begin{proposition}\label{principle}
    Let $A$ be a Noetherian commutative ring. Then, $A$ satisfies the $\splif$ principle $\Longrightarrow$ $A$ satisfies the $\sflilf$ principle $\Longrightarrow$ $A$ satisfies the $\splilf$ principle. 
\end{proposition}
The above proposition could have just been a one-sentence remark. But we believe it will be helpful for the reader to see a full proof because the proof is not immediately obvious, and it involves some useful play on our invariants exploiting some results from \cite{et2011} and \cite{abhs2011}.

\begin{lemma}\label{nonz-lem}
    Let $A$ be a commtutative ring, and let $G$ be a group.

    \begin{itemize}
        \item[(i)] Assume $A$ is $\aleph_0$-Noetherian. Then, $\sfli(AG)\leq \spli(AG)=\silp(AG)=\silf(AG)$.
        \item[(ii)] Assume $A$ is Noetherian and $G$ is finitely generated. Then,
        \begin{itemize}
            \item[(a)] $\sfli(AG)<\infty \Longleftrightarrow \silf(AG)<\infty$. In this case, $\sfli(AG)=\silf(AG)$.
            \item[(b)] $\sfli(AG)=\silf(AG)=\silp(AG)=\spli(AG)$.
        \end{itemize}
    \end{itemize}
\end{lemma}

\begin{proof}    \begin{itemize}
        \item[(i)] As noted earlier in the proof of Theorem \ref{m3}, $\spli(AG)=\silp(AG)$ here as $A$ is $\aleph_0$-Noetherian. The inequality follows as projectives are flat, and the rightmost equality is due to \cite[Proposition 2.1]{et2011}.

        \item[(ii)] \begin{itemize}
            \item[(a)] By (i), $\silf(AG)<\infty \Longrightarrow \sfli(AG)<\infty$. So, let $\sfli(AG)<\infty$, and now as $AG$ is Noetherian (due to $A$ being Noetherian and $G$ being finitely generated) and hence coherent, the argument from the proof of \cite[Theorem 3.7]{abhs2011} works (\cite[Theorem 3.7]{abhs2011} deals with $A=\mathbb{Z}$ with the hypothesis that $\mathbb{Z}G$ is coherent - the same argument works when $\mathbb{Z}$ is replaced with any commutative ring) to give us $\silf(AG)=\sfli(AG)$.

            \item[(b)] If $\silf(AG)$ is not finite, then by Part ii.(a), $\sfli(AG)$ is also not finite, and we are done by Part (i). And if $\silf(AG)$ is finite, we have $\silf(AG)=\sfli(AG)$ by Part ii.(a), and again, we are done by Part (i). 
        \end{itemize}
        \end{itemize}\end{proof}

We collect two easy consequences of Lemma \ref{nonz-lem}.
\begin{corollary}\label{spli-sfli-noetherian}
    \begin{itemize}
        \item[(i)] For any commutative Noetherian ring $A$, $\spli(A)=\sfli(A)$.
        \item[(ii)] For any group $G$, $G$ is locally finite $\Longleftrightarrow \sfli(\mathbb{Z}G)=1$. 
    \end{itemize}
\end{corollary}
\begin{proof} (i) Take $G$ to be the trivial group in Lemma \ref{nonz-lem}.ii.(b).

    (ii) Let $\sfli(\mathbb{Z}G)=1$, and let $H$ be a finitely generated subgroup of $G$. Then, $1=\sfli(\mathbb{Z})\leq \sfli(\mathbb{Z}H)\leq \sfli(\mathbb{Z}G)=1$. So, $\sfli(\mathbb{Z}H)=1$, and by Lemma \ref{nonz-lem}.ii.(b), $\spli(\mathbb{Z}H)=1$, and therefore $H$ must be finite by Proposition \ref{finchar}.(ii). The other direction is covered by \cite[Corollary 6.3]{ks}.
\end{proof}

\begin{remark}
     (i) Lemma \ref{nonz-lem}.ii.(b) can be thought of as a generalization of \cite[Corollary 3.9]{abhs2011}.

    (ii) Corollary \ref{spli-sfli-noetherian}.(i), which does not seem to have been known before, can be compared with \cite[Proposition 2.1]{et2011} which shows that ``$\silp$" and ``$\silf$" coincide for all rings.

    (iii) It follows from Corollary \ref{ghd0char} in Subsection \ref{s21} that $G$ is locally finite iff $\Ghd_{\mathbb{Z}}(G)=0$, so it might be tempting to think that one can deduce $\Ghd_{\mathbb{Z}}(G)=0$ from $\sfli(\mathbb{Z}G)=1$ and make Corollary \ref{spli-sfli-noetherian}.(ii) a direct consequence of Corollary \ref{ghd0char}, but unfortunately there does not seem to be an obvious way of making this deduction because Lemma \ref{lemma-ks}.(ii) only gives us $\Ghd_{\mathbb{Z}}(G)\leq 1$ from $\sfli(\mathbb{Z}G)=1$. 

    (iv) Note that it is a straightforward application of \cite[Lemma 2.10 and Corollary 6.3]{ks} and Corollary \ref{spli-sfli-noetherian}.(ii) that, for any group $G$, $G$ is locally finite $\Leftrightarrow$ $\sfli(AG)=\sfli(A)$ for all commutative rings $A$ $\Leftrightarrow$ $\sfli(\mathbb{Z}G)=1$. Similarly, as $\spli(AH)=\spli(AG)$ when $H$ is of finite index in $G$, $G$ is finite $\Leftrightarrow$ $\spli(AG)=\spli(A)$ for all commutative rings $A$ $\Leftrightarrow$ $\spli(\mathbb{Z}G)=1$.

\end{remark}

\begin{proof}[Proof of Proposition \ref{principle}] Let $G$ be a group and let $H$ be a finitely generated subgroup of it. 

\textbf{The first implication:} Assume that we have $\sfli(AG)=\sfli(A)$, so we have to now show that $H$ is finite under the extra assumption that $A$ satisfies the $\splif$ principle. 

By Lemma \ref{nonz-lem}.ii.(b), $\sfli(AH)=\silf(AH)$. We therefore have $\sfli(A)\leq \sfli(AH)\leq \sfli(AG)=\sfli(A)$, where the first inequality is always true. So, $\sfli(AH)=\sfli(A)$, and by Lemma \ref{nonz-lem}.ii.(b) and Corollary \ref{spli-sfli-noetherian}.(i), we then have $\spli(AH)=\sfli(AH)=\sfli(A)=\spli(A)$. Now, as $A$ satisfies the $\splif$ principle, $H$ must be finite.

\textbf{The second implication:} Assume that we have $\spli(AG)=\spli(A)$. We have $\sfli(A)\leq \sfli(AG)\leq \spli(AG)=\spli(A)=\sfli(A)$ using Corollary \ref{spli-sfli-noetherian}.(i) for the last equality. So, $\sfli(AG)=\sfli(A)$, and therefore $G$ has to be locally finite as $A$ satisfies the $\sflilf$ principle.\end{proof}

We end with the following question that is quite natural to ask - it remains open.

\begin{question}
    \begin{itemize}
        \item[(i)] Do all Noetherian commutative rings satisfy the $\splif$ (resp. the $\sflilf$ or the $\splilf$) principle?
        \item[(ii)] Is there a class of Noetherian commutative rings wherein a ring satisfies the $\splif$ principle if it is known to satisfy the $\sflilf$ principle?
        \item[(iii)] Are there any large classes of groups (likely defined in terms of Kropholler's hierarchy) such that all (Noetherian) commutative rings (if required, one can assume some finiteness conditions on these rings) satisfy the $\splif$ or the $\sflilf$ (or the weaker $\splilf$) principle as long as the groups are taken from these classes?
    \end{itemize}
\end{question}

\noindent \textit{Acknowledgements.} The first author is supported by the EPSRC Grant
EP/W036320/1 held by John Greenlees at the University of Warwick, and he would also like to thank the Isaac Newton Institute for Mathematical Sciences, Cambridge for their hospitality and support during the ``Equivariant homotopy theory in context" program during which a part of this work was undertaken. The second author is supported by the H.F.R.I. Grant ``GAAC", Project Number: 14907, awarded to Aristides Kontogeorgis at the University of Athens and implemented in the framework of H.F.R.I. call “Basic research Financing (Horizontal support of all Sciences)” under the National Recovery and Resilience Plan “Greece 2.0” funded by the European Union Next Generation EU, H.F.R.I..  
%Project Number: 14907.
% 16618 (MIS: 5163923)).
\begin{center}
	\includegraphics[scale=0.4]{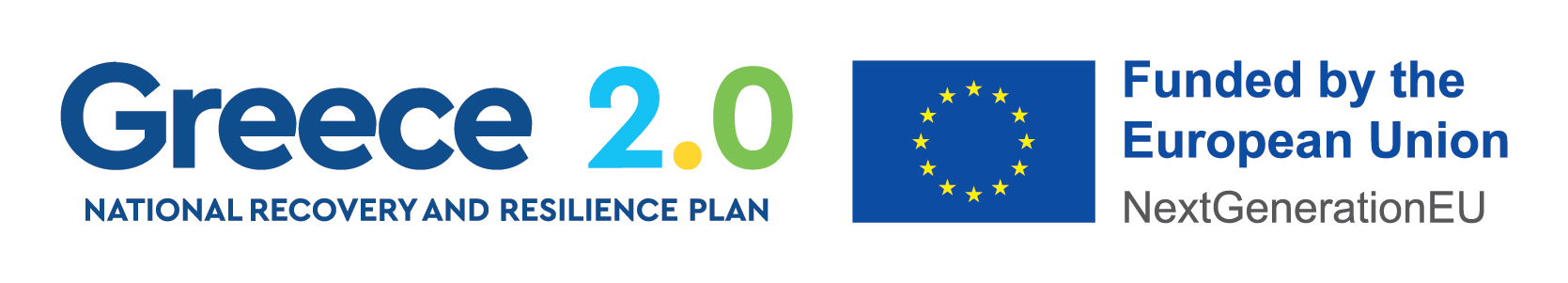}
	\hskip 1cm
	\includegraphics[scale=0.05]{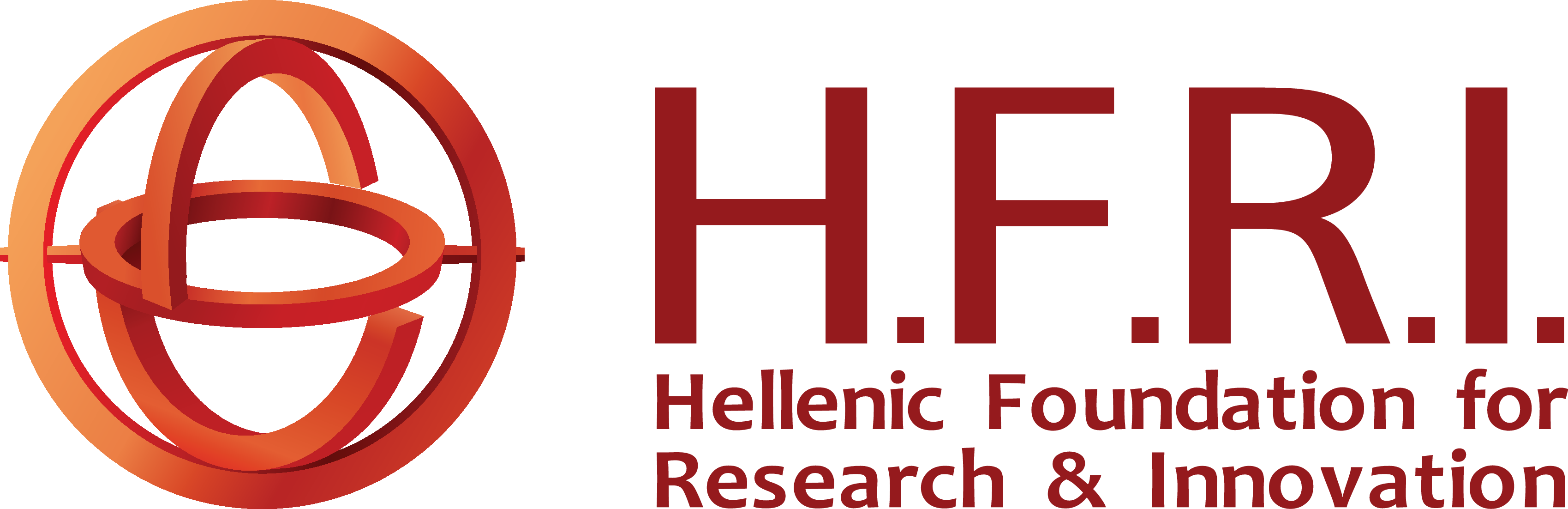}
\end{center}

\end{document}